\documentclass[leqno,11pt]{amsart}

\usepackage{amsfonts}
\usepackage{amsmath}
\usepackage{amsfonts}
\usepackage{amssymb}
\usepackage{amscd}

\newtheorem{theorem}{Theorem}

\newtheorem{corollary}[theorem]{Corollary}

\newtheorem{definition}[theorem]{Definition}

\newtheorem*{example*}{Example}

\newtheorem{lemma}[theorem]{Lemma}

\newtheorem{prop}[theorem]{Proposition}

\newcommand{\erre}{\mathbb{R}}
\newcommand{\p}{\mathbb{P}}

\newcommand{\acca}{\mathbb{H}}
\newcommand{\esse}{\mathbb{S}}

\newcommand{\ricc}{\operatorname{Ric}}

\newcommand{\sgn}{\operatorname{sgn}}

\newcommand{\hess}{\operatorname{hess}}
\newcommand{\Hess}{\operatorname{Hess}}

\newcommand{\di}{\mathrm{d}}

\newcommand{\ra}{\rightarrow}

\newcommand{\norm}[1]{{\|#1\|}}              
\newcommand{\pair}[1]{\langle#1\rangle}

\renewcommand{\hat}[1]{\widehat{#1}}

\begin{document}

\title[A sharp height estimate]{A sharp height estimate for compact hypersurfaces with constant $k$-mean curvature in warped product spaces}

\author[S. C. Garc\'ia-Mart\'inez]{Sandra C. Garc\'ia-Mart\'inez}
\address{Departamento de Matem\`aticas, Universidad de Murcia, Campus de Espinardo, 30100 Espinardo, Murcia, Spain.}
\email{scarolinagarciam@gmail.com}
\thanks{This work was partially supported by MINECO-FEDER project MTM2012-34037 and Fundaci\'{o}n S\'{e}neca
project 04540/GERM/06, Spain.
This research is a result of the activity developed within the framework of the Programme in Support of Excellence Groups of the Regi\'{o}n de Murcia, Spain, by Fundaci\'{o}n S\'{e}neca, Regional Agency
for Science and Technology (Regional Plan for Science and Technology 2007-2010). 
S. Carolina Garc\'\i a-Mart\'\i nez was supported by a research training grant within the framework of the programme
Research Training in Excellence Groups GERM by Universidad de Murcia. M. Rigoli was partially supported by MEC Grant SAB 2010-0073.}

\author[D. Impera]{Debora Impera}
\address{Dipartimento di Matematica e Applicazioni,
Universit\`a
degli studi di Milano Bicocca, via Cozzi 53, I-20125 Milano, Italy.}
\email{debora.impera@unimib.it}
\author[M. Rigoli]{Marco Rigoli}
\address{Dipartimento di Matematica,
Universit\`a
degli studi di Milano, via Saldini 50, I-20133 Milano, Italy.}
\email{marco.rigoli@unimi.it}
\date{\today}

\begin{abstract}
In this paper we obtain a sharp height estimate concerning compact hypersurfaces 
immersed into warped product spaces  
with some constant higher order mean curvature, 
and whose boundary is contained into a slice. We apply these results to draw topological conclusions 
at the end of the paper.
\end{abstract}

\maketitle

\section{Introduction}
In recent years height estimates for constant mean curvature graphs have been studied by several authors, since they are 
intimately related to important properties of the geometry of submanifolds. The first result in this 
direction is due to Heinz 
\cite{Heinz} who proved that a compact graph of positive constant mean curvature $H$ in Euclidean $3$-space with 
boundary on a plane can reach at most height $1/H$ from the plane. 
More recently an optimal bound was also obtained for compact graphs and also for compact embedded surfaces with constant mean curvature and 
boundary on a plane in the 3-dimensional hyperbolic space by N. Korevaar, R. Kusner, W. Meeks 
and B. Solomon \cite{KKMS}. In the case of a Riemannian product
$\erre\times\p^2$, with $\p^2$ any Riemannian surface, height estimates 
were exhibited by Hoffman, de Lira and Rosenberg in \cite{HLR} and by
Aledo, Espinar and G\'alvez in \cite{AEG}. 

The natural generalization of the mean curvature for an $n$-dimensional hypersurface are the $k$-mean 
curvatures $H_k$, $k=1,\ldots,n$, that are defined via the elementary symmetric functions of the principal curvatures of the immersion. 
Therefore, it is natural to extend the previous results to the case of constant higher order 
mean curvature. This was done first by Rosenberg in \cite{Rosenberg}, where he proved estimates for the 
height function of compact hypersurfaces with positive constant 
$k$-mean curvature $H_k$ embedded either into the euclidean or the hyperbolic space. 
Later, the same author and Cheng, \cite{ChengRosenberg}, found height estimates for compact vertical graphs with 
positive constant $k$-mean curvature in the product manifold 
$\erre\times\p$ and boundary on a slice, that is, on a submanifold of the form 
$\p_\tau=\{\tau\}\times\p$ for some $\tau\in\erre$. Some of the results in \cite{ChengRosenberg} where later 
improved by Espinar, G\'alvez and Rosenberg, \cite{EGR}, who proved sharp estimates and
horizontal estimates for positive extrinsic 
curvature surfaces in product manifolds. 
Finally, Al\'ias and Dajczer in \cite{aliasdajczer} gave height estimates in the case of 
compact hypersurfaces of positive constant mean curvature immersed into general warped product spaces 
and boundary on a slice, generalizing for $k=1$, the previous results obtained by Cheng and Rosenberg. Our aim 
is to complete the picture described above extending the results of Al\'ias and Dajczer to the case of compact hypersurfaces  
of constant positive $k$-mean curvature, $2\leq k\leq n$, in warped product spaces. 

We finally observe that, in \cite{ChengRosenberg} height estimates
were used to obtain topological properties of properly embedded hypersurfaces without 
boundary. More precisely, Cheng and Rosenberg proved that a hypersurface of constant $k$-mean curvature properly embedded in a 
product $\erre\times\p$, where $\p$ is a compact manifold with non-negative sectional curvature, has at least two ends, or, 
equivalently, it can not lie in a 
half-space. Furthermore, some topological restrictions using height and horizontal estimates were 
obtained in \cite{EGR}.

In this circle of ideas in the last part of this paper, we apply our height estimates in order to prove
non-existence results for properly immersed complete hypersurfaces without boundary in pseudo-hyperbolic spaces contained 
in a half-space.

\section{Preliminaries}
Let $f:\Sigma \rightarrow M^{n+1}$ be a connected hypersurface isometrically immersed into the Riemannian manifold
$M^{n+1}$. Let $A$ denote the second fundamental form of the immersion with respect to a (locally defined) normal
vector field $N$. Its eigenvalues,
$\kappa_1,\ldots,\kappa_n$, are the principal curvatures of the hypersurface (in the direction of $N$). Their elementary symmetric functions $S_k$, 
$k=0,...,n$, $S_0=1$, define the $k$-mean curvatures of the immersion via the formula
$$
H_k= {n \choose k}^{-1}S_k.
$$
Thus $H_1=H$ is the mean curvature, $H_n$ is the Gauss-Kronecker curvature and $H_2$ is, when the ambient space is
Einstein, a multiple of the scalar curvature up to an additive constant.

The Newton tensors associated to the immersion are inductively defined by
$$
P_0=I, \qquad P_k=S_k I-AP_{k-1}.
$$
For further use note that $\mathrm{Tr}P_k=(n-k)S_k$ and $\mathrm{Tr}AP_k=(k+1)S_{k+1}$.
In the sequel, we will need the operators $P_k$ to be globally defined on $T\Sigma$. Obviously, the second fundamental form
$A$ depends on the chosen local unit field $N$.  However, when $k$ is even the sign of $S_k$ (and hence $H_k$) does not
depend on the chosen $N$ which, by its very definition,  implies that the operator $P_k$ is a globally defined tensor field on $T\Sigma$. On the
other hand, when $k$ is odd in order to have $P_k$ globally defined, we need to assume that $\Sigma$ is
\textit{two-sided}. Recall that a hypersurface $f:\Sigma \rightarrow M^{n+1}$ is called two-sided if its normal bundle is trivial, i.e. there exists a
globally defined unit normal vector field $N$ on $\Sigma$. For instance, every hypersurface with never vanishing mean curvature is trivially
two-sided. 
In this last case, a choice of $N$ on $\Sigma$ makes the second fundamental form $A$ and its associated Newton tensors $P_k$ globally defined tensor fields
on $T\Sigma$.

Let $\nabla$ stand for the Levi-Civita connection of $\Sigma$. For a given function
$u\in C^{2}(\Sigma)$,
we denote by $\hess{u}:T\Sigma\rightarrow T\Sigma$ the symmetric operator given by $\hess{u}(X)=\nabla_X\nabla u$ for every $X\in T\Sigma$, and by
$\Hess{u}:T\Sigma\times T\Sigma\rightarrow C^{\infty}(\Sigma)$ the metrically equivalent bilinear form given by
\[
\Hess{u}(X,Y)=\pair{\hess{u}(X),Y}.
\]
Associated to each globally defined Newton tensor $P_k:T\Sigma\ra T\Sigma$, we consider the second order
differential operator $L_k:\mathcal{C}^{\infty}(\Sigma)\rightarrow\mathcal{C}^{\infty}(\Sigma)$ given by
$L_k=\mathrm{Tr}(P_k \circ \hess)$. In particular, $L_0$ is the Laplace-Beltrami operator $\Delta$. Observe that
\[
L_k(u)=\mathrm{div}(P_k\nabla u)-\pair{\mathrm{div}P_k,\nabla u},
\]
where $\mathrm{div}P_k=\mathrm{Tr}\nabla P_k$. This implies that $L_k$ is elliptic if and only if 
$P_k$ is positive definite and in this case the maximum principle holds for $L_k$. See for instance 
\cite[Theorem 3.1]{GT}.

Note that the ellipticity of the operator $L_1$ is guaranteed by the assumption $H_2>0$. Indeed, if this happens the mean curvature does not vanish on $\Sigma$ because of the basic inequality $H_1^2\geq H_2$. Therefore,
the immersion is two-sided and we can choose the normal unit vector $N$ on $\Sigma$ so that $H_1>0$. Furthermore, indicating with $k_j$ the eigenvalues of $A$,
\[
n^2H_1^2=\sum_{j=1}^n\kappa_j^2+n(n-1)H_2>\kappa_i^2
\]
for every $i=1,\ldots, n$, and then the eigenvalues of $P_1$, say $\mu_{1,i}$, satisfy $\mu_{1,i}=nH_1-\kappa_i>0$ for every $i$
(see, for instance, \cite[Lemma 3.10]{elbert}). This shows ellipticity of $L_1$. Regarding the ellipticity of $L_j$ when
$j\geq 2$, we assume the existence of an elliptic point in $\Sigma$, that is, a point 
at which the second fundamental form $A$
is positive definite with respect to an appropriate choice of $N$
. The existence of an elliptic point implies that $H_k$ is positive at that point, and applying G{\.a}rding
inequalities \cite{Ga}, we have
\begin{equation}
\label{garding}
H_1\geq H_2^{1/2}\geq\cdots\geq H_{k-1}^{1/(k-1)}\geq H_k^{1/k}>0,
\end{equation}
with equality at any stage only for an umbilical point. 
Therefore, in case $H_k$ is constant, the immersion is two-sided and $H_1>0$ for the chosen
orientation.
Moreover, in this case, the operators $L_{j}$ for every $1\leq j\leq k-1$ are elliptic or, equivalently, the
operators $P_j$ are positive definite (see \cite[Proposition 3.2]{barbosacolares}). Observe that the 
existence of an elliptic point is not guaranteed, in general, even in the compact case. For instance, 
it is clear that totally geodesic spheres and Clifford tori in $\esse^{n+1}$ are examples of compact 
isoparametric hypersurfaces without elliptic points. On the contrary, it is not difficult to see that 
every compact hypersurface in an open hemisphere has elliptic points (see the proof of Theorem 11.1 in \cite{aliasliramalacarne}).

In what follows, we consider the case when the ambient space is a warped product $M^{n+1}=I\times_{\rho} \p^n$,
where $I\subseteq\erre$ is an open interval, $\p^n$ is a complete $n$-dimensional Riemannian manifold and
$\rho:I \ra \erre_{+}$ is a smooth function.
The product manifold $I\times \p^n$ is endowed with the Riemannian metric
$$
\pair{,}=\pi_{I}^{*}(\di t^2)+\rho^2(\pi_{I})\pi_{\p}^*(\pair{,}_{\p}).
$$
Here $\pi_{I}$ and $\pi_{\p}$ denote the projections onto the corresponding factor and $\pair{,}_{\p}$ is the Riemannian
metric on $\p^n$. In particular, $M^{n+1}=I\times_{\rho} \p^n$ is complete if and only if $I=\erre$. We also observe that
each leaf $\p_t=\left\{t\right\} \times \p^n$ of the foliation $ t \ra \p_t$ of $M^{n+1}$ is a complete totally umbilical
hypersurface with constant $k$-mean curvature
$$
\mathcal{H}_k(t)=\Big( \frac{\rho'(t)}{\rho(t)}\Big)^k,\qquad 0 \leq k \leq n,
$$
with respect to the unit normal $T=-\partial/\partial t$.

Let $f:\Sigma \ra M^{n+1}=I \times_{\rho} \p^n$ be an isometrically immersed hypersurface.
We define the \textsl{height function}  $h \in C^{\infty}(\Sigma)$ by setting $h=\pi_{I} \circ f$. In this context and following the terminology
introduced in \cite{aliasdajczer2}, we will say that the hypersurface is \textsl{contained in a slab} if $f(\Sigma)$ lies between two leaves $\p_{t_1}, \p_{t_2}$ with $t_1<t_2$ of
the foliation. Finally, we define the \textit{angle function} $\Theta:\Sigma\rightarrow[-1,1]$ by
\[
\Theta=\pair{N,T}. 
\]
It is useful to recall that $\norm{\nabla h}^2=1-\Theta^2$. Furthermore we observe that if $f$ is locally a graph over $\p^n$ (that is, transversal to $T$) then either $\Theta<0$ or $\Theta>0$ along $f$. Thus, the assumption that $\Theta$ has constant sign is generally weaker than that of $f$ being a local graph. This assumption will be used in many of the results below.

We conclude this section stating a computational result that will be useful in order to prove our main theorems.
For a proof see \cite[Proposition 6]{aliasimperarigoli} and \cite[Lemma 27]{aliasimperarigoli}.

\begin{lemma}
\label{lemmasigmahtheta}
Let $f: \Sigma \ra M^{n+1}=I \times_{\rho} \p^n$ be an isometric immersion into a warped product space and let $h$ be the height function.
\begin{itemize}
 \item[$(i)$] Define
$$
\sigma(t)=\int_{t_0}^t \rho(u) \di u.
$$
Then
\begin{equation}\label{hessh}
\Hess h(X,Y)=\mathcal{H}(h)(\pair{X,Y}-\pair{\nabla h,X}\pair{\nabla h,Y})+\Theta \pair{AX,Y},
\end{equation}
for each pair of vector fields $X$ and $Y$ on $T\Sigma$,
\begin{equation}\label{lrh}
L_k h=\mathcal{H}(h)(c_kH_k-\pair{P_k\nabla h,\nabla h})+c_k\Theta H_{k+1},
\end{equation}
and
\begin{equation}\label{lrsigma}
L_k \sigma(h)=c_k\rho(h)(\mathcal{H}(h)H_k+\Theta H_{k+1}),
\end{equation}
where $c_k=(n-k){n \choose k}=(k+1){n \choose {k+1}}$, $\mathcal{H}(t)=\rho'(t)/\rho(t)$.\\
\item[$(ii)$]
Let $\hat{\Theta}=\rho \Theta$. Then
\begin{align*}
L_k \hat{\Theta}=&-{n \choose {k+1}}\rho(h)\pair{\nabla h, \nabla H_{k+1}}-\rho'(h)c_k H_{k+1}\\
&-\hat{\Theta}\mathcal{H}'(h)(\norm{\nabla h}^2c_kH_k-\pair{P_k \nabla h,\nabla h})-\frac{\hat{\Theta}}{\rho(h)^2}\beta_k\\
&-\hat{\Theta}{n \choose{k+1}} (nH_1H_{k+1}-(n-k-1)H_{k+2}),
\end{align*}
where
$$
\beta_k=\sum_{i=1}^n\mu_{k,i}K_{\p}({\pi_{\p}}_*E_i,{\pi_{\p}}_*N)\norm{{\pi_{\p}}_*E_i\wedge {\pi_{\p}}_*N}^2.
$$
Here $K_{\p}$ denotes the sectional curvature of $\p$ and the $\mu_{k,i}$'s stand for the eigenvalues of $P_k$  with respect to
a local orthonormal frame $\{E_1,\ldots,E_n\}$ diagonalizing $A$ (and hence $P_k$).
\end{itemize}
\end{lemma}
\section{Height estimates in warped products and pseudo-hyperbolic spaces}
The aim of this section is to prove height estimates for compact hypersurfaces in warped product spaces with boundary contained in a slice. We begin by introducing a well-known tangency principle due to Fontenele and Silva \cite{fontenelesilva}. We recall first the following definition taken from \cite{fontenelesilva}.

\begin{definition}
Let $\Sigma_1$ and $\Sigma_2$ be hypersurfaces of $M^{n+1}$ that are tangent at $p$, i.e. which satisfy $T_p\Sigma_1=T_p\Sigma_2$. Fix a unitary vector $\eta_0$ normal to $\Sigma_1$ at $p$. We say that $\Sigma_1$ remains above $\Sigma_2$ in a neighbourhood of $p$ with respect to $\eta_0$ if, when we parametrize $\Sigma_j$, $j=1,2$, on a small neighbourhood $W_j$ of $0$ in $T_p\Sigma_j$ by 
\[
\varphi_j(x)=\exp_p(x+\mu_j(x)\eta_0),
\] 
for some smooth function $\mu_j:W_j\rightarrow\mathbb{R}$ satisfying $\mu_j(0)=0$, the functions $\mu_j$ satisfy $\mu_1(x)\geq\mu_2(x)$ in a sufficiently small neighbourhood of $0\in T_p\Sigma_1$.
\end{definition}
Observe that the fact that 
$\Sigma_1$ lies above $\Sigma_2$ in a neighbourhood of a point $p$ is equivalent to requiring that 
the geodesics of $M$ that are normal to the hypersurface which is totally geodesic at $p$ 
(namely $\mathrm{exp}_p(W)$), in a neighbourhood of $p$ intercept $\Sigma_2$ before $\Sigma_1$.

For hypersurfaces, Fontenele and Silva \cite{fontenelesilva} proved the following tangency principle

\begin{theorem}[Theorem 1.1 in \cite{fontenelesilva}]\label{tanprin}
Let the hypersurfaces $\Sigma_1$ and $\Sigma_2$ of $M^{n+1}$ be tangent at an interior point $p$ and let $\eta_0$ be a unitary vector that is normal to $\Sigma_1$ at $p$. Suppose that $\Sigma_1$ remains above $\Sigma_2$ in a neighbourhood of $p$ with respect to $\eta_0$. Denote by $H_k^{\Sigma_1}$ and $H_k^{\Sigma_2}$ the $k$-mean curvatures of $\Sigma_1$ and $\Sigma_2$, respectively. Assume that 
for some $k$, $1\leq k\leq n$, we have $H_k^{\Sigma_2}\geq H_k^{\Sigma_1}$
in a neighbourhood of zero and, if $k\geq 2$, that $\lambda^{\Sigma_2}(0)$, the principal curvature vector of $\Sigma_2$ at zero, belongs to $\Gamma_k$, that is to the connected component in $\erre^n$ of the set $\{S_k>0\}$ containing $(1,\ldots,1)$. Then $\Sigma_1$ and $\Sigma_2$ coincide in a neighbourhood of $p$. 
\end{theorem}  

Note that if there exists an elliptic point of $\Sigma_2$ and $H_k^{\Sigma_2}$ is constant, the assumption $\lambda^{\Sigma_2}(0)\in\Gamma_k$ is 
automatically satisfied. 

Theorem \ref{tanprin} enable us to extend Proposition 3.7 in \cite{aliasdajczer} to the case of hypersurfaces with constant $k$-mean curvature, for any $1\leq k\leq n$.

\begin{prop}\label{proptanprin}
Let $f:\Sigma\ra\erre\times_\rho\p^n$ be a compact constant $k$-mean curvature hypersurface, $1\leq k\leq n$, with boundary 
$f(\partial\Sigma)\subset\p_\tau$ for some $\tau\in\erre$.  
Then the following holds:
\begin{itemize}
\item[$(i)$] If $H_k\leq\inf_{[\tau,+\infty)}\mathcal{H}_k$ and $\mathcal{H}(\tau)>0$, $\mathcal{H}'\geq0$ on $[\tau,+\infty)$ for $k\geq2$, then $h\leq \tau$;
\item[$(ii)$] If $H_k\geq \sup_{(-\infty,\tau]}\mathcal{H}_k$ and either $H_2>0$ or there exists an elliptic point of $\Sigma$ when $k\geq3$ then $h\geq\tau$.
\end{itemize}
\end{prop}
\begin{proof}
Let us prove $(i)$ first. Assume that $H_k\leq\inf_{[\tau,+\infty)}\mathcal{H}_k$, but that $h\leq\tau$ is false. Then, by compactness of $\Sigma$, there exists a point $p_0\in\Sigma$ where $h$ 
attains its maximum, that is
\[\max_\Sigma h=h(p_0)=\tau_0\]
and
\[\tau<\tau_0.\] 
We consider $\Sigma_1=\Sigma$ and $\Sigma_2=\p_{\tau_0}$. Note that $\Sigma_1\neq\Sigma_2$, 
$\Sigma_1$ and $\Sigma_2$ are tangent at the common point $p_0$ and $\Sigma_1$ remains above 
$\Sigma_2$ with respect to the normal $\eta_0=-T$ at $p_0$. Moreover,
\[
H_k^{\Sigma_1}=H_k\leq\inf_{[\tau,+\infty)}\mathcal{H}_k\leq(\mathcal{H}(\tau_0))^k=H_k^{\Sigma_2}.
\]
Furthermore, at $\tau_0$, because of the assumptions on $\mathcal{H}$, $\mathcal{H}(\tau_0)>0$, so that $p_0$ is an elliptic point for $\Sigma_2$. 
We thus conclude using Theorem \ref{tanprin} that $\Sigma_1$ and $\Sigma_2$ must coincide in a neighborhood of $p_0$. Using the fact that $\Sigma$ is connected we easily contradict $\Sigma_1\neq\Sigma_2$.

In case $(ii)$, we reason again by contradiction and assume that $h\geq\tau$ is false. Hence, again by compactness of $\Sigma$, there exists a point $p_1\in\Sigma$ where $h$ 
attains its minimum, that is
\[\min_\Sigma h=h(p_1)=\tau_1.\]
Hence 
\[\tau>\tau_1\] 
and we choose $\Sigma_1=\p_{\tau_1}$ and $\Sigma_2=\Sigma$. Again  $\Sigma_1\neq\Sigma_2$, $\Sigma_1$ and $\Sigma_2$ are tangent at the common point $p_0$ and $\Sigma_1$ remains above $\Sigma_2$ with respect to the normal $\eta_0=-T$. Furthermore 
\[
H_k^{\Sigma_2}=H_k\geq\sup_{(-\infty,\tau]}\mathcal{H}_k\geq\mathcal{H}_k(\tau_1)=H_k^{\Sigma_1}.
\]
Because of the assumptions on $\Sigma_2=\Sigma$ we conclude again by applying the tangency principle. 
\end{proof}

With a different approach and slightly different assumptions we now prove the following
\begin{prop}\label{propheight}
Let $f:\Sigma\ra\erre\times_\rho\p^n$ be a compact constant $k$-mean curvature hypersurface, $1\leq k\leq n$, with boundary 
$f(\partial\Sigma)\subset\p_\tau$ for some $\tau\in\erre$, whose angle function $\Theta$ does not change sign. 
Assume that $\mathcal{H}'\geq0$ and that $\mathcal{H}(h)\neq0$. Then either $h\leq \tau$ or $h\geq\tau$.
\end{prop}
\begin{proof}
Assume by contradiction that there exist points $p_0$, $p_1$ of $\Sigma$ satisfying
\[
\tau_0=h(p_0)=\max_\Sigma h>\tau>\min_\Sigma h=h(p_1)=\tau_1.
\]
In case $\mathcal{H}(h)>0$ let us choose the orientation so that $\Theta\leq0$.
Then, since $p_0$ is a point of maximum for $h$ in the interior of $\Sigma$, using \eqref{hessh} for each $v\in T_{p_0}\Sigma$ we have
\begin{equation}\label{hessellipt}
0\geq \Hess h(p_0)(v,v)=\mathcal{H}(\tau_0)\norm{v}^2+\Theta(p_0)\pair{Av,v}(p_0)>-\pair{Av,v}(p_0).
\end{equation}
Hence $p_0$ is an elliptic point and, as already observed in the introduction, each $L_j$ is elliptic for any $1\leq j\leq k-1$, $H_k$ is positive and
\[
H_1\geq H_2^{1/2}\geq\cdots\geq H_{k-1}^{1/(k-1)}\geq H_k^{1/k}>0.
\]
For $1\leq k\leq n$, we consider the operator
\[
\mathcal{L}_{k-1}=\mathrm{Tr}
\Big(\Big[\sum_{j=0}^{k-1}(-1)^j\frac{c_{k-1}}{c_j}\mathcal{H}(h)^{k-1-j}\Theta^{j}P_j\Big]\circ \hess \Big)=
\mathrm{Tr}(\mathcal{P}_{k-1}\circ\hess),
\]
where
\[
\mathcal{P}_{k-1}=\sum_{j=0}^{k-1}(-1)^j\frac{c_{k-1}}{c_j}\mathcal{H}(h)^{k-1-j}\Theta^{j}P_j.
\]
Using induction on $k$ it is not difficult to prove (see \cite[Section 6]{aliasimperarigoli}) that
$$
\mathcal{L}_{k-1}\sigma(h)=c_{k-1}\rho(h)(\mathcal{H}(h)^{k}+(-1)^{k-1}\Theta^{k}H_{k}).
$$
Moreover, the previous observations imply that
$\mathcal{L}_{k-1}$ is a semi-elliptic operator and, since $\Theta(p_0)=-1=\Theta(p_1)$, we have
\[
\mathcal{L}_{k-1}\sigma(h)(p_{0})=c_{k-1}\rho(\tau_0)(\mathcal{H}(\tau_0)^k-H_k)\leq 0
\]
and
\[
\mathcal{L}_{k-1}\sigma(h)(p_{1})=c_{k-1}\rho(\tau_1)(\mathcal{H}(\tau_1)^k-H_k)\geq 0.
\]
Thus, since $\mathcal{H}(h)> 0$ on $\Sigma$,
$$\mathcal{H}(\tau_1)\geq H_k^{1/k}\geq \mathcal{H}(\tau_0).$$ 
Immediately, the assumption on $\mathcal{H}'(h)$ implies that
$\mathcal{H}(h)=H_k^{1/k}$ is constant on $\Sigma$. Therefore, by the G{\.a}rding inequality
$H_{k-1}\geq H_k^{(k-1)/k}$ and the fact that $\Theta\geq -1$, using 	\eqref{lrsigma} we obtain
\begin{eqnarray*}
L_{k-1} \sigma(h) & = & c_{k-1}\rho(h)H_k^{1/k}(H_{k-1}+\Theta H^{(k-1)/k}_k)\\
{} & \geq & c_{k-1}\rho(h)H_k^{1/k}(H_{k-1}-H^{(k-1)/k}_k)\geq 0.
\end{eqnarray*}
That is, $L_{k-1}\sigma(h)\geq 0$ on the compact manifold $\Sigma$. Therefore, by the maximum principle applied to the
elliptic operator $L_{k-1}$ we conclude that $\sigma(h)$, and hence $h$, must attain its maximum on $\partial\Sigma$;  contradiction.

Finally, when $\mathcal{H}(h)<0$, we can choose the orientation so that $\Theta\geq 0$ and consider then the operator 
\[
\mathcal{\hat{L}}_{k-1}=\mathrm{Tr}
\Big(\Big[\sum_{j=0}^{k-1}(-1)^{k-1-j}\frac{c_{k-1}}{c_j}\mathcal{H}(h)^{k-1-j}\Theta^{j}P_j\Big]\circ \hess \Big)=
\mathrm{Tr}(\mathcal{\hat{P}}_{k-1}\circ\hess),
\]
where
\[
\mathcal{\hat{P}}_{k-1}=\sum_{j=0}^{k-1}(-1)^{k-1-j}\frac{c_{k-1}}{c_j}\mathcal{H}(h)^{k-1-j}\Theta^{j}P_j.
\]
Then, similarly to what we did above, since $p_1$ is a point of minimum for $h$ in the interior of $\Sigma$, it holds
\[
0\leq \Hess h(p_1)(v,v)=\mathcal{H}(\tau_1)\norm{v}^2+\Theta(p_1)\pair{Av,v}(p_1)<\pair{Av,v}(p_1).
\] 
Hence $p_1$ is an elliptic point and $\mathcal{\hat{L}}_{k-1}$ is semi-elliptic. Furthermore, it is not difficult to prove using induction on $k$ that
$$
\mathcal{\hat{L}}_{k-1}\sigma(h)=c_{k-1}\rho(h)((-1)^{k-1}\mathcal{H}(h)^{k}+\Theta^{k}H_{k}).
$$
The proof then follows as in the case $\Theta\leq 0$.
\end{proof}

In the sequel, we will focus on hypersurfaces of constant $k$-mean curvature in pseudo-hyperbolic ambient spaces. 
Following the terminology 
introduced by Tashiro in \cite{tashiro}, we say that a manifold is \textit{pseudo-hyperbolic} if it is a warped product $\erre\times_\rho\p^n$ where the warping
function is a solution for some $c<0$ of the ordinary differential equation
\begin{equation}\label{oderho}
\rho''+c\rho=0.
\end{equation}
Thus either $\rho(t)=\cosh(\sqrt{-c}t)$ or $\rho(t)=\mathrm{e}^{\sqrt{-c}t}$. The terminology is due to the fact that with a suitable choice of the fiber, we obtain a representation of the hyperbolic space. Indeed, for more details we refer, for instance, to Montiel \cite{montiel}, the hyperbolic space $\acca^{n+1}$ can be viewed as a hypersphere in the Lorentz-Minkowski space, precisely, as a connected component of the hyperquadric
\[
\{p\in\erre^{n+2}_1|\pair{p,p}=-1\}.
\]
If we fix $a\in\erre^{n+2}_1$ and consider the closed (see \cite{montiel} for this notion) conformal vector field
\[
\mathcal{T}(p)=a+\pair{a,p}p,\qquad p\in\acca^{n+1},
\]
depending on the causal character of $a$ we have different foliations of $\acca^{n+1}$ and hence different descriptions of it as a warped product space. Namely, if $a$ is timelike, $\acca^{n+1}$ is foliated by spheres and can be described as the warped product 
$\erre_+\times_{\sinh t}\esse^n$; if $a$ is lightlike the space is foliated by horospheres and it can be viewed as 
$\erre\times_{\mathrm{e}^t}\erre^n$; finally, if $a$ is spacelike, the vector field $\mathcal{T}$ generates a foliation of 
$\acca^{n+1}$ by means of totally geodesic hyperbolic hyperplanes and it can be represented as the warped product
$\erre\times_{\cosh t}\acca^n$.

In the rest of the paper we will restrict ourselves to the case $c=-1$ in \eqref{oderho}. In case $\rho(t)=\mathrm{e}^t$ the following consequence of Proposition \ref{proptanprin} is straightforward. 
\begin{corollary}\label{coroet}
Let $f:\Sigma\ra\erre\times_{\mathrm{e}^t}\p^n$ be a compact constant $k$-mean curvature hypersurface, $1\leq k\leq n$, with boundary 
$f(\partial\Sigma)\subset\p_\tau$
for some $\tau\in\erre$. Then
\begin{itemize}
\item[$(i)$] if  $H_k\leq1$ then $h\leq\tau$; 
\item[$(ii)$] if  $H_k\geq1$ and there exists an elliptic point when $k\geq3$, then $h\geq\tau$.
\end{itemize}
In particular, $H_k=1$ if and only 
if $h=\tau$.
\end{corollary}

In case $\rho(t)=\cosh t$, arguing as in the proof of Proposition \ref{proptanprin}, we prove the following
\begin{corollary}\label{corocosht}
Let $f:\Sigma\ra\erre\times_{\cosh t}\p^n$ be a compact constant $k$-mean curvature hypersurface, $1\leq k\leq n$, with boundary 
$f(\partial\Sigma)\subset\p_0$ and such that the angle function $\Theta$ does not change sign.
\begin{enumerate}
\item[$(i)$] If $H_k\leq0$, then $h\leq 0$.\\
\item[$(ii)$] If either $H_k\geq0$ when $k$ is odd or $H_k\geq1$ when $k$ is even, then $h\geq 0$.
\end{enumerate}
\end{corollary}
\begin{proof}
Let us prove $(i)$ first. Assume $H_k\leq0$ and suppose by contradiction that $h\leq0$ is false. Then there exists a point $p_0\in\Sigma$ where the function $h$ attains its maximum and such that 
$\tau_0=h(p_0)>0$. Since $\mathcal{H}(h)=\tanh h$, it follows that $\mathcal{H}(\tau_0)>0$. Besides, it is clear that 
$$
0=\inf_{[0,+\infty)}\mathcal{H}_k(t)=\inf_{[0,+\infty)}(\tanh t)^k\leq(\tanh \tau_0)^k
$$ 
and we may apply the tangency principle as in Proposition \ref{proptanprin} to arrive to a contradiction.

Similarly, in case $(ii)$, assume by contradiction that $h\geq0$ is false. By the compactness of $\Sigma$ there exists a point $p_1$ where the function $h$ attains its minimum $\tau_1=h(p_1)<0$. Reasoning as above, if we prove that there exists an elliptic point on $\Sigma$, we can apply the same argument as in the proof of Proposition \ref{proptanprin} in order to obtain a contradiction. Concerning the existence of the elliptic point, since $p_1$ is a point of minimum for $h$, note that
using \eqref{hessh} and recalling that $\Theta(p_1)\pm1$, we have
\[
0\leq \Hess h(p_1)=\tanh (\tau_1)\norm{v}^2+\Theta(p_1)\pair{Av,v}(p_1)< \sgn \Theta \pair{Av,v}(p_1).
\]
Hence, if $\Theta\geq0$, and we can always reduce ourselves to this case by changing orientation when $k$ is even, the point $p_1$ is an elliptic point. On the other hand, if $\Theta\leq0$ and necessarily $k$ is odd, then
\[
\pair{Av,v}(p_1)<0.
\]
This gives a contradiction to $H_k\geq0$.
\end{proof}

In \cite{aliasdajczer} Al\'ias and Dajczer, using the tangency principle and exploiting the subharmonicity of the function $\sigma(h)H_1+\rho(h)\Theta$, were 
able to obtain the following height estimates for constant mean curvature hypersurfaces in pseudo-hyperbolic spaces.
\begin{theorem}[Theorem 3.9 in \cite{aliasdajczer}]\label{Theorem3.9AD}
Let $f:\Sigma\ra\erre\times_{\mathrm{e}^t}\p^n$ be a compact constant mean curvature  hypersurface satisfying $H_1\notin[0,1)$, with boundary 
$f(\partial\Sigma)\subset\p_\tau$ for some $\tau\in\erre$, whose angle function $\Theta$ does not change sign. Assume that $\ricc_\p\geq0$ and set 
$C=\log\big(H_1/(H_1-1)\big)$. Then,
\begin{enumerate}
\item if $H_1>1$ then $\tau\leq h\leq\tau+C$; 
\item if $H_1<0$ then $\tau+C\leq h\leq\tau$;
\item $H_1=1$ if and only if $h=\tau$ and $f(\Sigma)$ is contained in the slice $\p_\tau$.
\end{enumerate}
\end{theorem}

\begin{theorem}[Theorem 3.10 in \cite{aliasdajczer}]\label{Theorem3.10AD} 
Let $f:\Sigma\ra\erre\times_{\cosh t}\p^n$ be a compact hypersurface of constant mean curvature $H_1$.
Suppose that the boundary of $\Sigma$ satisfies $f(\partial\Sigma)\subset\p_0$ and that the angle function $\Theta$ does not change sign. 
Assume that $\ricc_\p\geq-1$ and set $\tanh C=1/H_1$. Then, 
\begin{enumerate}
\item if $H_1>1$ then $0\leq h\leq C$; 
\item if $H_1\leq-1$ then $C\leq h\leq 0$;
\item if $H_1=0$ then $h=0$ and $f(\Sigma)$ is contained in the slice $\p_0$.
\end{enumerate}
\end{theorem}

Using the results of the previous section we extend Theorems 3.9 and 3.10 in \cite{aliasdajczer} to 
constant higher order mean curvature hypersurfaces. 
We first need the following
\begin{prop}\label{subharm}
Let $f:\Sigma\ra\erre\times_{\rho}\p^n$ be a compact hypersurface of constant $k$-mean curvature $H_k\geq1$ in a pseudo-hyperbolic space with boundary $\partial \Sigma\subset \p_\tau$, for some $\tau\in\erre$, with $\tau=0$ when $\rho(t)=\cosh t$. Assume that the angle function $\Theta$ does not change sign and that, if $k\geq3$, there exists an elliptic point on $\Sigma$. If
\[
K_\p\geq\sup_\erre\{-\rho^2\mathcal{H}'\},
\] 
then $\Theta\leq0$ and the function $\phi=\sigma(h)H_k^{1/k}+\rho(h)\Theta$ satisfies
\[
L_{k-1}\phi\geq0.
\]
\end{prop}   
\begin{proof}
Since $H_k\geq1$, Corollaries \ref{coroet} and \ref{corocosht} imply that $h\geq\tau$. In particular, $\mathcal{H}(h)\geq0$. 
Since $\Sigma^n$ is compact, there exists a point $p_0 \in \Sigma$ 
where the height function attains its maximum. Then $\nabla h(p_0)=0$, $\Theta(p_0)=\pm 1$ and by \eqref{lrh}
$$
0\geq\Delta h(p_0)=n\mathcal{H}(\tau_0)+n\Theta(p_0)H_1(p_0)> n\Theta(p_0)H_1(p_0).
$$
Then $\Theta(p_0)=-1$ and $\Theta$ is non-positive since we assumed that it does not change sign.

Let us prove that $L_{k-1} \phi \geq 0$.
Since $H_k$ is constant, recalling that $\hat{\Theta}=\rho\Theta$ and using Equation \eqref{lrsigma} and $(ii)$ of Lemma \ref{lemmasigmahtheta}, we have
\begin{align*}
L_{k-1}\phi=&H_k^{\frac 1k}L_{k-1} \sigma(h)+L_{k-1} \hat{\Theta} \\
=& H_k^{\frac 1k}c_{k-1}(\rho'(h)H_{k-1}+\hat{\Theta}H_k)-c_{k-1} H_{k-1}\hat{\Theta}
\norm{\nabla h}^2\mathcal{H}'(h)\\
&+ \hat{\Theta}\mathcal{H}'(h)\pair{P_{k-1} \nabla h, \nabla h}-\hat{\Theta}{n \choose k} (nH_1H_k-(n-k)H_{k+1})\\
&-\rho'(h)c_{k-1} H_k- \frac{\hat{\Theta}}{\rho(h)^2}\sum_{i=1}^n\mu_{k-1,i}K_{\p}({\pi_{\p}}_*E_i,{\pi_{\p}}_*N)\norm{{\pi_{\p}}_*E_i\wedge{\pi_{\p}}_*N}^2\\
=&A+B+C,
\end{align*}
where
$$
A=-\hat{\Theta}{n \choose k} (nH_1H_k-(n-k)H_{k+1}-kH_k^{\frac{k+1}{k}}),
$$
$$
B=c_{k-1}\rho'(h)(H_{k-1}H_k^{\frac 1k}-H_k)
$$
and
\begin{align*}
C=&-\hat{\Theta}\mathcal{H}'(h)(\norm{\nabla h}^2c_{k-1}H_{k-1}-\pair{P_{k-1} \nabla h,\nabla h})\\
&-\frac{\hat{\Theta}}{\rho(h)^2}\sum_{i=1}^n\mu_{k-1,i}K_{\p}({\pi_{\p}}_*E_i,{\pi_{\p}}_*N)\norm{{\pi_{\p}}_*E_i\wedge{\pi_{\p}}_*N}^2.
\end{align*}
Then by G{\.a}rding inequalities,
$$
H_{k-1}H_k^{\frac 1k}-H_k=H_k^{\frac 1k}(H_{k-1}-H_k^{\frac{k-1}{k}}) \geq 0.
$$
Moreover,
$$nH_1H_k-kH_k^{\frac{k+1}{k}} \geq nH_k^{\frac{k+1}{k}}-kH_k^{\frac{k+1}{k}}=(n-k)H_k^{\frac{k+1}{k}},
$$
hence
$$
nH_1H_k-kH_k^{\frac{k+1}{k}}-(n-k)H_{k+1} \geq (n-k)(H_k^{\frac{k+1}{k}}-H_{k+1})\geq 0.
$$
Finally, set $\alpha:=\sup_{\erre} \{{\rho'}^2-\rho''\rho\}$. Since
$$
\norm{{\pi_{\p}}_*E_i\wedge{\pi_{\p}}_*N}^2=\norm{\nabla h}^2-\pair{E_i,\nabla h}^2,
$$
taking into account that the $\mu_{k-1,i}$'s are positive, we have
\begin{align*}
&\sum_{i=1}^n\mu_{k-1,i}K_{\p}({\pi_{\p}}_*E_i,{\pi_{\p}}_*N)\norm{{\pi_{\p}}_*E_i\wedge{\pi_{\p}}_*N}^2\\
& \geq \alpha \sum_{i=1}^n\mu_{k-1,i}\norm{{\pi_{\p}}_*E_i\wedge{\pi_{\p}}_*N}^2\\
&=\alpha(c_{k-1}H_{k-1}\norm{\nabla h}^2-\pair{P_{k-1}\nabla h,\nabla h}).
\end{align*}
Hence,
\begin{align*}
\frac{1}{\rho(h)^2}&\sum_{i=1}^n\mu_{k-1,i}K_{\p}({\pi_{\p}}_*E_i,{\pi_{\p}}_*N)\norm{{\pi_{\p}}_*E_i\wedge{\pi_{\p}}_*N}^2\\
&+\mathcal{H}'(h)(\norm{\nabla h}^2c_{k-1}H_{k-1}-\pair{P_{k-1} \nabla h,\nabla h})\\
\geq & \Big(\frac{\alpha}{\rho(h)^2}+\mathcal{H}'(h)\Big)(\norm{\nabla h}^2c_{k-1}H_{k-1}-\pair{P_{k-1} \nabla h,\nabla h})
\geq 0,
\end{align*}
where the last inequality follows from $\alpha=\sup_\erre\{-\rho^2\mathcal{H}'\}$ and from the fact that $P_{k-1}$ is a positive definite operator. 
\end{proof}

Using the previous lemma we prove the following results.

\begin{theorem}\label{ocho}
Let $f:\Sigma\ra\erre\times_{\mathrm{e}^t}\p^n$ be a compact constant $k$-mean curvature hypersurface, $2\leq k\leq n$, with boundary 
$f(\partial\Sigma)\subset\p_\tau$ for some $\tau\in\erre$ and whose angle function $\Theta$ does not change sign. Assume that $K_\p\geq0$ and set 
$C=\log\big(H_k^{1/k}/(H_k^{1/k}-1)\big)$.
If  $H_k>1$ and there exists an elliptic point when $k\geq3$, then $\tau\leq h\leq\tau+C$. 
Furthermore, $H_k=1$ if and only if $h=\tau$.
\end{theorem}
\begin{proof}
By Corollary \ref{coroet}, when $H_k>1$ we observe that $h\geq\tau$ and $H_k=1$ if and only if $h\equiv\tau$. Furthermore, the existence of the elliptic point and the fact that $\Theta$ does not change sign, together with $K_\p\geq0$, imply by Proposition \ref{subharm} that $\Theta\leq0$ and 
\[
L_{k-1}\phi\geq 0,
\]
where $\phi=\mathrm{e}^h(H_k^{1/k}+\Theta)$. Since $L_{k-1}$ is, in the present situation, elliptic, the classical maximum principle implies
\[ 
\mathrm{e}^h(H_k^{1/k}-1)\leq\phi\leq\max_{\partial\Sigma}\phi=\mathrm{e}^{\tau}(H_k^{1/k}+\max_{\partial\Sigma}\Theta)\leq\mathrm{e}^{\tau}H_k^{1/k}.
\]
Therefore, if $H_k>1$ we have
\[
\mathrm{e}^h\leq\mathrm{e}^\tau\frac{H_k^{1/k}}{H_k^{1/k}-1},
\]
and the conclusion follows at once.

\end{proof}

\begin{theorem}\label{mainthmcosh} 
Let $f:\Sigma\ra\erre\times_{\cosh t}\p^n$ be a compact hypersurface of constant $k$-mean curvature $H_k\geq1$, $2\leq k\leq n$.
Suppose that the boundary of $\Sigma$ satisfies $f(\partial\Sigma)\subset\p_0$ and that the angle function $\Theta$ does not change sign. 
Assume that $K_\p\geq-1$ and set $\tanh C=1/H_k^{1/k}$. Then $0\leq h\leq C$.
\end{theorem}
\begin{proof}
Since $H_k\geq 1>0$, by Corollary \ref{corocosht} the height function $h$, and hence the function $\mathcal{H}(h)$, has to be non-negative. Let us denote by $p_0$ 
the point where $h$ attains its maximum. If $\Theta\leq0$, reasoning as in the Proof of Proposition \ref{propheight}, that is, using the differential inequality \eqref{hessellipt}, we find that there exists a point 
where all the principal curvatures are positive, indeed $\mathcal{H}(h(p_0))>0$. If $k$ is even we can always choose the orientation on $\Sigma$ so that $\Theta\leq0$. On the other hand, if $k$ is odd and $\Theta\geq0$, then
$$
0\geq \Hess h(p_0)(v,v)=\tanh (h(p_0))\pair{v,v}+\Theta(p_0)\pair{Av,v}(p_0)> 
\pair{Av,v}(p_0),
$$
for any $v\neq0$, contradicting the assumption on the sign of $H_k$. Summarizing, we can assume without loss on generality that  
$\Theta\leq0$. Then $p_0$ is an elliptic point on $\Sigma$, $H_j>0$ and each $L_{j-1}$ is elliptic for any $1\leq j\leq k$. 
Applying Proposition \ref{subharm}, we obtain that the function $\phi=\sinh h H_k^{1/k}+\cosh h \Theta$ is subharmonic with 
respect to $L_{k-1}$. Then, by the classical maximum principle
\[
\phi\leq\max_{\partial\Sigma}\phi=\sinh 0\, H_k^{1/k}+\cosh 0\,\max_{\partial\Sigma}\Theta\leq 0.
\]
On the other hand
\[
\sinh h\,H_k^{1/k}-\cosh h\leq \phi,
\]
so that, from the above,
\[
\sinh h\,H_k^{1/k}\leq\cosh h.
\]
Dividing both sides by $\cosh h$, we obtain
\[
\tanh h \leq \frac{1}{H_k^{1/k}},
\]  
that is, the desired conclusion.
\end{proof}
Observe that in case $\p^n\equiv\acca^n$, we recover a result obtained by Rosenberg when $\Sigma$ is a graph 
(see \cite[Theorem 6.2]{Rosenberg}).
\begin{corollary}
Let $f:\Sigma\ra \acca^{n+1}$ be a compact hypersurface of constant $k$-mean curvature $H_k\geq1$, $1\leq k\leq n$, 
with boundary $f(\partial\Sigma)\subset\acca^n$ and whose angle function $\Theta$ does not change sign. 
Then $|h|\leq \mathrm{arctanh}(1/H_k^{1/k})$.
\end{corollary}
\begin{proof}
Simply recall the description $\acca^{n+1}=\erre\times_{\cosh t}\acca^n$ given above.
\end{proof}
\section{Geometric applications}
In \cite{HLR} height estimates for compact graphs of constant mean curvature in three-dimensional 
product manifolds are used in order to obtain some information on the topology at infinity of 
properly embedded surfaces of constant mean curvature. The proofs of these results are essentially 
based on the use of the Alexandrov reflection principle, exploiting the fact that horizontal 
translations are isometries, as well as reflection through each $\p_\tau$. The same technique is 
used in \cite{ChengRosenberg} to study the topology at infinity of hypersurfaces of constant higher 
order mean curvature properly embedded in $(n+1)$-dimensional product manifolds. Unfortunately the 
same technique is not applicable to hypersurfaces in warped products spaces since the Alexandrov 
reflection principle does not work in this context. Nevertheless, using the height estimates that we 
have found in the previous section, we are still able to prove topological results for non-compact 
hypersurfaces of constant mean and higher 
order mean curvature in pseudo-hyperbolic spaces in the same spirit of those described above, even 
replacing the assumption of embeddedness as in \cite{ChengRosenberg} by one on the 
angle function.\\

We begin introducing the following
\begin{definition}
Let $\Sigma$ be a hypersurface in a warped product space $\erre\times_\rho\p^n$. We say that $\Sigma$ \textit{lies in an upper or lower half-space} if it is respectively contained in a region of $\erre\times_\rho\p^n$ of the form
\[
[a,+\infty)\times\p^n\quad\text{or}\quad(-\infty,a]\times\p^n,
\]
for some real number $a$.
\end{definition}
For what it is concerned with hypersurfaces of constant mean curvature in $\erre\times_{\mathrm{e}^t}\p^n$ we prove the following
\begin{theorem}\label{thmendh}
Let $\Sigma$ be a non-compact hypersurface (without boundary) of 
constant mean curvature $H_1\notin [0,1)$, properly immersed in a pseudo-hyperbolic space $\erre\times_{\mathrm{e}^t}\p^n$. 
Suppose that the fiber $\p^n$ is compact and has Ricci curvature satisfying $\ricc_\p\geq 0$ and that the angle function $\Theta$ does 
not change sign. 
Then the hypersurface can not lie in a half-space. In particular, $\Sigma$ must have at least one top and one bottom end. 
\end{theorem}
\begin{proof}
Suppose that $\Sigma$ lies in a half-space of the form $(-\infty,a]\times\p^n$, $a\in\erre$. 
For any $\tau\in\erre$, $\tau<a$, denote by $\Sigma_\tau$ the hypersurface
\[
\Sigma_{\tau}=\{(t,x)\in\Sigma \, |\, t\geq\tau\}.
\]
Observe that $\Sigma_{\tau}$ is contained in a slab of width $a-\tau$. Moreover, as before, since $\p$ is compact and the immersion is proper, 
$\Sigma_{\tau}$ is compact with boundary contained in $\p_{\tau}$. Furthermore, it satisfies the assumptions of Theorem 
\ref{Theorem3.9AD} and we conclude that either $H_1<0$ and $h\leq\tau$, giving a contradiction, or $H_1\geq1$ and 
$\Sigma_{\tau}$ is contained in a slab of width $\log\big(H_1/(H_1-1)\big)$, so that it must be $a-\tau\leq\log\big(H_1/(H_1-1)\big)$. Choosing $\tau$ sufficiently small we violate this 
estimate, reaching a contradiction. 
On the other hand, if $\Sigma$ is contained in a half-space of the form $[a,+\infty)\times\p^n$, 
for any 
$\tau\in\erre$, $\tau>a$, denote by $\Sigma_{\tau}$ the hypersurface
\[
\Sigma_{\tau}=\{(t,x)\in\Sigma \,| \,t\leq\tau\}.
\]
Observe that $\Sigma_{\tau}$ is contained in a slab of width $\tau-a$. Moreover, since $\p$ is compact and the 
immersion is proper, $\Sigma_{\tau}$ is compact with boundary contained in $\p_{\tau}$. Furthermore, it satisfies the assumptions of 
Theorem \ref{Theorem3.9AD} and we can conclude that $\Sigma_{\tau}$ is contained in a slab of width $\log\big(H_1/(H_1-1)\big)$. 
Choosing $\tau$ sufficiently large we violate this estimate and get to a contradiction. 
Finally, if $H_1\geq1$ we obtain a contradiction again since Theorem \ref{Theorem3.9AD} implies that $h\geq\tau$.  
\end{proof}
The previous theorem extends to hypersurfaces of constant higher order mean curvature as follows.
\begin{theorem}\label{thmendhk}
Let $\Sigma$ be a non-compact hypersurface (without boundary) properly immersed in a pseudo-hyperbolic space 
$\erre\times_{\mathrm{e}^t}\p^n$ with compact fiber $\p^n$ satisfying $K_\p\geq 0$. Assume that $\Sigma$ has constant $k$-mean 
curvature $H_k> 1$, $2\leq k\leq n$, that there exists an elliptic point if $k\geq3$ and that the angle function $\Theta$ does 
not change sign. 
Then the hypersurface can not lie in a half-space. In particular, $\Sigma$ must have at least one top and one bottom end. 
\end{theorem}

\begin{proof}
Suppose
that $\Sigma$ lies in a half-space of the form $(-\infty,a]\times\p^n$, $a\in\erre$. 
For any $\tau\in\erre$, $\tau<a$, denote by $\Sigma_\tau$ the hypersurface
\[
\Sigma_\tau=\{(t,x)\in\Sigma \, |\, t\geq\tau\}.
\]
Observe that $\Sigma_\tau$ is contained in a slab of width $a-\tau$. Furthermore, since $\p$ is compact and the immersion is proper, 
$\Sigma_\tau$ is compact with boundary contained in $\p_\tau$. Moreover, it satisfies the assumptions of Theorem \ref{ocho} and we
conclude that $\Sigma_\tau$ is contained in a slab of width $\log\big(H_k^{1/k}/(H_k^{1/k}-1)\big)$. Choosing $\tau$ sufficiently small
we violate this estimate, and get a contradiction. 
On the other hand, if $\Sigma$ is contained in a half-space of the form $[a,+\infty)\times\p^n$, then for any $\tau\in\erre$, $\tau>a$, 
we introduce the hypersurface $\Sigma_\tau$ defined by
\[
\Sigma_\tau=\{(t,x)\in\Sigma \, |\, t\leq\tau\}.
\]
Thus $\Sigma_\tau$ is a compact hypersurface with boundary on $\p_\tau$ and contained in a slab of width $\tau-a$. 
Theorem \ref{ocho} imply then that $h\geq\tau$ giving a contradiction.  
\end{proof}

Applying Proposition \ref{proptanprin}, we shall now prove the following result on the topology at infinity of
hypersurfaces of constant mean curvature in a pseudo-hyperbolic space $\erre\times_{\cosh t}\p^n$.
\begin{theorem}
Let $\Sigma$ be a non-compact hypersurface (without boundary) of constant mean curvature, 
properly immersed in a pseudo-hyperbolic space $\erre\times_{\cosh t}\p^n$ with compact fiber $\p^n$.
\begin{enumerate}
\item if $H_1\geq1$, then $\Sigma$ cannot lie in an upper half-space, that is, $\Sigma$ must have at least a bottom end;
\item if $H_1\leq-1$, then $\Sigma$ cannot lie in a lower half-space, that is, $\Sigma$ must have at least a top end.
\end{enumerate}
\end{theorem}
\begin{proof}
Assume first that $H_1\geq 1$ and suppose by contradiction that $\Sigma$ lies in an upper half-space $[a,+\infty)\times\p^n$. For a fixed $\tau>a$ consider the hypersurface
\[
\Sigma_\tau=\{(t,x)\in\Sigma \ |\ t\leq \tau\}.
\]
Then $\Sigma_\tau$ is a compact hypersurface of constant mean curvature $H_1$ and with boundary contained in the slice $\p_\tau$. Since
\[
\sup_{(-\infty,\tau]}\mathcal{H}_1=\sup_{t\in(-\infty,\tau]} \tanh t=\tanh \tau, 
\] 
it follows that $H_1\geq1\geq\tanh \tau$ and we can apply Proposition \ref{proptanprin} to obtain that $h\geq\tau$, leading to a 
contradiction. 

On the other hand, if $H_1\leq-1$, assume by contradiction that $\Sigma$ lies in a lower half-space $(-\infty,a]\times\p^n$. 
Consider the hypersurface
\[
\Sigma_\tau=\{(t,x)\in\Sigma \ |\ t\geq \tau\}.
\]
for a fixed $\tau<a$. 
Then $\Sigma_\tau$ is a compact hypersurface of constant mean curvature $H_1$ and with boundary contained in the slice $\p_\tau$.
Moreover
\[
H_1\leq-1\leq\tanh \tau=\inf_{t\in[\tau,+\infty)}\tanh t=\inf_{[\tau,+\infty)}\mathcal{H}_1, 
\] 
and we conclude the proof as above applying Proposition \ref{proptanprin}.
\end{proof}
Finally, the previous theorem can be extended to the case of constant $k$-mean curvature as follows:
\begin{theorem}
Let $\Sigma$ be a non-compact hypersurface (without boundary) of constant positive $k$-mean curvature, 
$k\geq2$, properly immersed in a pseudo-hyperbolic space $\erre\times_{\cosh t}\p^n$ with compact fiber $\p^n$. If $k>2$ assume that there exists an 
elliptic point of $\Sigma$. 
If $H_k\geq1$ then $\Sigma$ cannot lie in an upper half-space, that is, $\Sigma$ must have at least a bottom end.
\end{theorem}
\begin{proof}
Assume 
that $\Sigma$ lies in an upper half-space $[a,+\infty)\times\p^n$. Since $H_k\geq1$ and there exists an elliptic point on $\Sigma$, we have that
each $H_j$ is positive on $\Sigma$ and each $L_j$ is elliptic for any $1\leq j\leq k-1$. For a fixed $\tau>a$ consider the hypersurface
\[
\Sigma_\tau=\{(t,x)\in\Sigma \ |\ t\leq \tau\}.
\]
Then $\Sigma_\tau$ is a compact hypersurface of constant $k$-mean curvature $H_k\geq1$ and with boundary contained in the slice $\p_\tau$. Moreover, since
\[
\sup_{(-\infty,\tau]}\mathcal{H}_k=\sup_{t\in(-\infty,\tau]} (\tanh t)^k\leq1\leq H_k
\] 
we can apply Proposition \ref{proptanprin} and obtain that $h\geq\tau$, which leads to a contradiction. 
\end{proof}

\bigskip

\bibliographystyle{amsplain}
\bibliography{biblioSHE}

\providecommand{\bysame}{\leavevmode\hbox to3em{\hrulefill}\thinspace}
\providecommand{\MR}{\relax\ifhmode\unskip\space\fi MR }
\providecommand{\MRhref}[2]{%
  \href{http://www.ams.org/mathscinet-getitem?mr=#1}{#2}
}
\providecommand{\href}[2]{#2}
\begin{thebibliography}{10}

\bibitem{AEG}
J.A. Aledo, J.M. Espinar, and J.A. G{\'a}lvez, \emph{Height estimates for
  surfaces with positive constant mean curvature in {$\Bbb M^2\times\Bbb R$}},
  Illinois J. Math. \textbf{52} (2008), no.~1, 203--211.

\bibitem{aliasdajczer2}
L.J. Al{\'{\i}}as and M.~Dajczer, \emph{Uniqueness of constant mean curvature
  surfaces properly immersed in a slab}, Comment. Math. Helv. \textbf{81}
  (2006), no.~3, 653--663.

\bibitem{aliasdajczer}
L.J. Al\'ias and M.~Dajczer, \emph{Constant mean curvature hypersurfaces in
  warped product spaces}, Proc. Edinb. Math. Soc. (2) \textbf{50} (2007),
  511--526.

\bibitem{aliasliramalacarne}
L.J. Al{\'{\i}}as, J.H.S. de~Lira, and J.M. Malacarne, \emph{Constant
  higher-order mean curvature hypersurfaces in {R}iemannian spaces}, J. Inst.
  Math. Jussieu \textbf{5} (2006), no.~4, 527--562.

\bibitem{aliasimperarigoli}
L.J. Al{\'{\i}}as, D.~Impera, and M.~Rigoli, \emph{Hypersurfaces of constant
  higher order mean curvature in warped products}, to appear on Trans. Amer.
  Math. Soc.

\bibitem{barbosacolares}
J.L.M. Barbosa and A.G. Colares, \emph{Stability of {H}ypersurfaces with
  {C}onstant $r$-{M}ean {C}urvature}, Ann. Glob. An. Geom. \textbf{15} (1997),
  277--297.

\bibitem{ChengRosenberg}
X.~Cheng and H.~Rosenberg, \emph{Embedded positive constant {$r$}-mean
  curvature hypersurfaces in {$M^m\times\bold R$}}, An. Acad. Brasil. Ci\^enc.
  \textbf{77} (2005), no.~2, 183--199.

\bibitem{elbert}
M.F. Elbert, \emph{Constant positive 2-mean curvature hypersurfaces}, Illinois
  J. Math. \textbf{46} (2002), no.~1, 247--267.

\bibitem{EGR}
J.~M. Espinar, J.~A. G{\'a}lvez, and H.~Rosenberg, \emph{Complete surfaces with
  positive extrinsic curvature in product spaces}, Comment. Math. Helv.
  \textbf{84} (2009), no.~2, 351--386.

\bibitem{fontenelesilva}
F.~Fontenele and S.L. Silva, \emph{A tangency principle and applications},
  Illinois J. Math. \textbf{45} (2001), no.~1, 213--228.

\bibitem{Ga}
L.~G{\.a}rding, \emph{An inequality for hyperbolic polynomials}, J. Math. Mech.
  \textbf{8} (1959), 957--965.

\bibitem{GT}
D.~Gilbarg and N.S. Trudinger, \emph{Elliptic partial differential equations of
  second order}, second ed., Grundlehren der Mathematischen Wissenschaften
  [Fundamental Principles of Mathematical Sciences], vol. 224, Springer-Verlag,
  Berlin, 1983.

\bibitem{Heinz}
E.~Heinz, \emph{On the nonexistence of a surface of constant mean curvature
  with finite area and prescribed rectifiable boundary}, Arch. Rational Mech.
  Anal. \textbf{35} (1969), 249--252.

\bibitem{HLR}
D.~Hoffman, J.H.S. de~Lira, and H.~Rosenberg, \emph{Constant mean curvature
  surfaces in {$M^2\times\bold R$}}, Trans. Amer. Math. Soc. \textbf{358}
  (2006), no.~2, 491--507 (electronic).

\bibitem{KKMS}
N.J. Korevaar, R~Kusner, W.~H. Meeks, III, and B.~Solomon, \emph{Constant mean
  curvature surfaces in hyperbolic space}, Amer. J. Math. \textbf{114} (1992),
  no.~1, 1--43.

\bibitem{montiel}
S.~Montiel, \emph{Unicity of constant mean curvature hypersurfaces in some
  {R}iemannian manifolds}, Indiana Univ. Math. J. \textbf{48} (1999), no.~2,
  711--748.

\bibitem{Rosenberg}
H.~Rosenberg, \emph{Hypersurfaces of constant curvature in space forms}, Bull.
  Sci. Math. \textbf{117} (1993), no.~2, 211--239.

\bibitem{tashiro}
Y.~Tashiro, \emph{Complete {R}iemannian manifolds and some vector fields},
  Trans. Amer. Math. Soc. \textbf{117} (1965), 251--275.

\end{thebibliography}

\end{document}